\documentclass[a4paper]{article}
\usepackage{amssymb}
\usepackage{amsmath}
\usepackage{amsthm}
\usepackage{hyperref}
\usepackage{enumerate}
\usepackage{mathtools}
\usepackage{url}
\usepackage[T1]{fontenc}

\newtheorem{Theorem} {Theorem} [section]
\newtheorem{Proposition} [Theorem] {Proposition}
\newtheorem{Lemma} [Theorem] {Lemma}
\newtheorem{Corollary} [Theorem] {Corollary}
\newtheorem{Example} [Theorem] {Example}

\newcommand{\Ff}{{\mathbb F}}

\newcommand{\cB}{{\mathcal B}}
\newcommand{\cC}{{\mathcal C}}

\newcommand{\cF}{{\mathcal F}}

\newcommand{\cM}{{\mathcal M}}
\newcommand{\cP}{{\mathcal P}}

\newcommand{\cR}{{\mathcal R}}
\newcommand{\cS}{{\mathcal S}}
\newcommand{\cT}{{\mathcal T}}

\newcommand{\PGL}{\mathrm{PGL}}
\newcommand{\GL}{\mathrm{GL}}

\newcommand{\Stab}{\mathrm{Stab}}
\newcommand{\<}{\langle}
\renewcommand{\>}{\rangle} 
\renewcommand{\phi}{\varphi} 

\title{On MSR Subspace Families of Lines}
\author{
Ferdinand Ihringer
}
\date{24 October 2023}

\begin{document}
\maketitle

\begin{abstract}
  A minimum storage regenerating (MSR) 
  subspace family of $\Ff_q^{2m}$ is a set $\cS$ of $m$-spaces
  in $\Ff_q^{2m}$ such that for any $m$-space $S$ in $\cS$ there 
  exists an element in $\PGL(2m, q)$ which maps $S$ to a complement
  and fixes $\cS \setminus \{ S \}$ element-wise.
  We show that an MSR subspace family of $2$-spaces in $\Ff_q^4$
  has at most size $6$ with equality if and only if it is 
  a particular subset of a Segre variety.
  This implies that an $(k{+}2, k, 4)$-MSR code has $k \leq 7$.
\end{abstract}

\section{Introduction}

Distributed storage systems (DSS) require codes which are particularly 
suited for dealing with the unavailability of some storage nodes.
This leads to the investigation of \textit{minimum storage regenerating (MSR)
codes}. It has been shown that MSR codes are closely linked
to MSR subspace families, cf.~\cite{AG2021,TWB2014,WTB2012},
which we will define below. Here we obtain a precise upper bound
on the size of $(k+2,k,4)$-MSR codes, to our knowledge the first open case,
using techniques from finite geometry.

Let $\GL(N, q)$ denote the general linear group over the field with $q$ elements,
that is the set of intervertible matrices over $\Ff_q$. 
We denote the multiplicative group of $\Ff_q$ by $\Ff_q^*$.
Let $\PGL(N, q) = \GL(N, q)/\Ff_q^*$
denote the corresponding projective general linear group.
In the below we use exponential notation for group actions.

Consider $\Ff_q^{rm}$.
We want to find a family $\cF$ of 
$m$-spaces $S_i$ such that there exist 
$g_{i,j} \in \PGL(rm, q)$, where $j \in \{ 1, \ldots, r-1 \}$, such that for all distinct $i,i'$,
we have that $S_i + S_i^{g_{i,1}} + \ldots + S_i^{g_{i,r-1}} = \Ff_q^{rm}$ and
that $S_{i'}^{g_{i,j}} = S_{i'}$. We call a family $\{ S_1, \ldots, S_k \}$ such that 
such $g_{i,j}$ exist a \textit{minimum storage regenerating (MSR) subspace family}
or $(rm, r)$-MSR subspace family, cf.\ \cite{AG2021}.
Note that in \cite{AG2021} only $g_{i,j} \in \mathrm{GL}(rm, q)$ is required,
but this makes no difference as we only consider actions on subspaces.
Whenever we limit ourselves to the case that $r=2$,
we write $g_i$ instead of $g_{i,1}$
and $S_i + S_i^{g_{i}} = \Ff_q^{2m}$ is 
equivalent to $S_i \cap S_i^{g_i}$ being trivial.
In this document matrices act by left multiplication.
Throughout the document, let $e_i$ denote 
the $i$-th vector of the canonical basis.
We use projective notation and call $1$-spaces \textit{points},
$2$-spaces \textit{lines}, and $3$-spaces \textit{planes}.
We also say that two subspace are disjoint if their intersection 
is $\{ 0 \}$.

Theorem 3 of \cite{AG2021} states that an $(rm, r)$-MSR
subspace family has size at most $2 \ln(rm)/\ln(\frac{r^2}{r^2-r+1})$.
A detailed 
analysis of the proof of Theorem 3 of \cite{AG2021} 
gives an upper bound of $7$,
see \S\ref{sec:rthree}.
Our main result is the following:

\begin{Theorem}\label{thm:MSR_n_2}
  A $(4,2)$-MSR subspace family has size at most $6$. In case of 
  equality, $q \geq 3$ and the family is isomorphic to
  \begin{align*}
      &\< e_1, e_3 \>, &&\< e_2, e_4 \>, &&\< e_1 {+} e_2, e_3{+}e_4\>,\\
      &\< e_1, e_2 \>, &&\< e_3, e_4 \>, &&\< e_1 {+} e_3, e_2{+}e_4\>.
  \end{align*}
\end{Theorem}

Finally, let us sketch the connection to MSR codes.
We follow (the far more detailed) discussion 
which is given in \cite{AG2021}.

Recall that a {\it maximum distance seperable (MDS) code} $\cC$ 
over a field $\Ff$ with parameters $(n, k, \ell)$ 
corresponds to the vectors in a $\Ff$-linear subspace 
of $(\Ff^\ell)^n$ such that any $k$ entries of $c \in \cC$
determine $c$. Note that 
a coordinate of $c$ corresponds to a vector in $\Ff^\ell$.
An \textit{$(n, k, \ell)$-MSR code} 
is an $(n, k, \ell)$-MDR code where
for every $m \in \{ 1, \ldots, k \}$ there exist 
linear functions $h_{i,m}: \Ff^\ell \rightarrow \Ff^{\ell/r}$
such that the code $m$-th entry of $c \in C$ can be computed
by an $\Ff$-linear operation on 
$\< h_{i,m}(c_i): i \in \{ 1, \ldots, n \} 
\setminus \{ m \} \> \in \Ff^{(n-1)\ell/r}$.

Now let $\cC \subseteq (\Ff^\ell)^n$ be an $(n, k, \ell)$-MSR code 
with redundancy $r := n-k$. As we can reconstruct a $c \in \cC$
from its first $k$ symbols, we find invertible matrices 
$C_{i,j} \in \Ff^{\ell \times \ell}$ for $i \in \{ 1, \ldots, r \}$
and $j \in \{ 1, \ldots, k \}$ such that 
\[
  c_{k+i} = \sum_{j=1}^k C_{i,j} c_j.
\]
For $c_m$ with $m \in \{ 1, \ldots, k \}$ we find matrices 
$S_{1,m}, \ldots, S_{r,m} \in \Ff^{\ell/r \times \ell}$ such that 
we can recover $c_m$ from $S_{i,m} C_{i,m} c_{m}$.
In particular one can see, cf.~\cite{AG2021}, that 
the $\ell \times \ell$ matrix $(S_{i,m}C_{i,m})_i$
has full rank.

We say that an MSR code possesses {\it constant repair matrices}
if $S_{i',m} = S_{i,m}$ for all $i,i' \in \{ 1, \ldots, r \}$.
Theorem 2 of \cite{TWB2014} implies that if there 
exists an $(n, k, \ell)$-MSR code, then there also exists 
an $(n-1, k-1, \ell)$-MSR code with constant repair matrices.
Now consider an $(n, k, \ell)$-MSR code with 
constant repair matrices in the same notation as above.
Then we obtain an $(n, r)$-MSR subspace family:
For $S_i$ take the row span of $S_{i,m}$.
For the $g_{i,j}$ take the map of the row 
vector $x$ of $S_{i,m}$ to $x C_{j+1,m} C_{1,m}^{-1}$.
Thus, Proposition 2 of \cite{AG2021} follows:

\begin{Proposition}
 If there exists an $(n, k, \ell)$-MSR code over a field $\Ff$,
 then there exists an $(\ell, r)$-MSR subspace family with $k-1$ subspaces.
\end{Proposition}


Therefore, our main result implies

\begin{Corollary}
  A $(k+2,k,4)$-MSR code has $k \leq 7$. 
\end{Corollary}

\section{Constructions}

A group $G$ acts \textit{regularly} on a set $S$ if for any $x,y\in S$
there exists a unique $g \in G$ with $x^g = y$. 
A \textit{(projective) frame} or \textit{projective basis}
is a tuple of $N+1$ points of $\Ff_q^N$ of which all 
$N$-element subsets span $\Ff_q^N$. We use repeatedly 
the well-known fact that $\PGL(N, q)$ acts regularly on frames, that is 
for two frames $\cB = (B_1, \ldots, B_{N+1})$,
$\cB' = (B_1', \ldots, B_{N+1}')$,
there exists a unique $g \in \PGL(N, q)$ such that
$B_i^g = B_i'$ for all $i \in \{ 1, \ldots, N+1 \}$, cf.\ \cite[p.\ 19]{Taylor91}.
Let us give a construction which is essentially identical 
to the construction given in Appendix A in \cite{AG2021} for $r=2$.
We denote the identity matrix by $I$.

\begin{Example}\label{ex:constr}
  Let $q \geq 3$ and let $\alpha$ be an element of $\Ff_q \setminus \{ 0, 1\}$.
  The following are MSR subspace families.
  \begin{enumerate}
   \item For $(2,1)$-MSR subspace family, take
   \begin{align*}
      &S_1 = \< e_1 \>, &&S_2=\< e_2 \>, &&S_3 = \< e_1 + e_2\>,\\
      &g_1 = \begin{pmatrix}
             1 & 0 \\ 
             \alpha & 1-\alpha 
            \end{pmatrix},
      &&g_2 = \begin{pmatrix}
             1-\alpha^{-1} & \alpha^{-1} \\
             0 & 1
            \end{pmatrix},
      &&g_3 = \begin{pmatrix}
             1 & 0\\
             0 & \alpha
            \end{pmatrix}.
   \end{align*}
   In particular, $S_i^{g_i} = \< e_1 + \alpha e_2 \>$.
   \item For a $(4, 2)$-MSR subspace family, take
   {
   \footnotesize
   \begin{align*}
      &S_1 = \< e_1, e_3 \>, &&S_2=\< e_2, e_4 \>,\\ &S_3 = \< e_1 {+} e_2, e_3{+}e_4\>,
      &&S_4 = \< e_1, e_2 \>, \\ &S_5=\< e_3, e_4 \>, &&S_6 = \< e_1 {+} e_3, e_2{+}e_4\>,\\
      \intertext{ and }
      &g_1 = \begin{pmatrix}
             1 & 0 & 0 & 0 \\ 
             \alpha & 1{-}\alpha & 0 & 0 \\
             0 & 0 & 1 & 0 \\
             0 & 0 & \alpha & 1{-}\alpha
            \end{pmatrix},
      &&g_2 = \begin{pmatrix}
             1{-}\alpha^{-1} & \alpha^{-1} & 0 & 0\\
             0 & 1 & 0 & 0 \\
             0 & 0 & 1{-}\alpha^{-1} & \alpha^{-1} \\
             0 & 0 & 0 & 1
            \end{pmatrix},\\
      &g_3 = \begin{pmatrix}
             1 & 0 & 0 & 0 \\
             0 & \alpha & 0 & 0\\
             0 & 0 & 1 & 0 \\
             0 & 0 & 0 & \alpha
            \end{pmatrix},
      &&g_4 = \begin{pmatrix}
             1 & 0 & 0 & 0 \\ 
             0 & 1 & 0 & 0 \\
             \alpha & 0 & 1{-}\alpha & 0 \\
             0 & \alpha & 0 & 1{-}\alpha
            \end{pmatrix},\\
      &g_5 = \begin{pmatrix}
             1{-}\alpha^{-1} & 0 & \alpha^{-1} & 0 \\
             0 & 1{-}\alpha^{-1} & 0 & \alpha^{-1} \\
             0 & 0 & 1 & 0 \\
             0 & 0 & 0 & 1
            \end{pmatrix},
      &&g_6 = \begin{pmatrix}
             1 & 0 & 0 & 0 \\
             0 & 1 & 0 & 0 \\
             0 & 0 & \alpha & 0\\
             0 & 0 & 0 & \alpha
            \end{pmatrix}.
   \end{align*}\par}
   In particular, $S_i^{g_i} = \< e_1 + \alpha e_2, e_3+ \alpha e_4 \>$ for $i \in \{ 1, 2, 3\}$,
   and $S_i^{g_i} = \< e_1 + \alpha e_3, e_2+ \alpha e_4 \>$ for $i \in \{ 4, 5, 6\}$.
   \item 
Let $\cF = \{ S_i \}$ be an $(2m, m)$-MSR subspace family of $\Ff_q^{2m}$ of size $k$.
Then we obtain an MSR subspace family $\cF' = \{ S_i' \}$ of size $k+3$ in $\Ff_q^{4m}$ as follows:
Put 
\begin{align*}
  &S_{k+1}' = \< e_1, \ldots, e_{2m} \>, \\
  & S_{k+2}' = \< e_{2m+1}, \ldots, e_{4m} \>, \\
  & S_{k+3}' = \< e_1{+}e_{2m+1}, \ldots, e_{2m}{+}e_{4m} \>,\\
  & g_{k+1}' = 
  \begin{pmatrix}
    I & 0\\
    \alpha I & (1-\alpha)I
  \end{pmatrix},\\
  & g_{k+2}' = 
  \begin{pmatrix}
    (1-\alpha^{-1})I & \alpha^{-1} I \\
    0 & I
  \end{pmatrix},\\
  & g_{k+3}' = 
  \begin{pmatrix}
    I & 0 \\
    0 & \alpha I
  \end{pmatrix}.
\end{align*}
As $\PGL(4m, q)$ acts regularly on frames,
there exists a unique element $h \in \allowbreak \PGL(4m, q)$ which maps $e_i$ to $e_{i+2m}$ for 
all $i \in \{ 1, \ldots, 2m \}$. Then put 
\begin{align*}
  & S_i' = S_i + S_i^h, 
  && g_i' = 
  \begin{pmatrix}
    g_i & 0 \\
    0 & g_i
  \end{pmatrix}.
\end{align*}
  \end{enumerate}
\end{Example}

Note that the second example can be obtained 
from the first example using the recursion of the third example.
The recursion is a tensor product construction, hence
this clarifies the aforementioned connection to the construction 
in Appendix A of \cite{AG2021}.

\section{Upper Bounds}

The idea for our proof of Theorem \ref{thm:MSR_n_2} 
is to generalize the following proof for MSR subspace 
families in $\Ff_q^2$ to $\Ff_q^4$. Call an 
MSR subspace family \textit{maximal} if it is not 
contained in a larger MSR subspace family.

For a subspace $S$ of $\Ff_q^N$, let $\Stab(S)$ denote 
the stabilizer of $S$ in $\PGL(N, q)$. For us, usually
$S$ will be an $N/2$-space. Furthermore,
for a family $\cS = \{ S_1, \ldots,\allowbreak  S_k \}$ of subspaces of $\Ff_q^m$,
let $\Stab(\cS) = \bigcap_{i=1}^k \Stab(S_i)$ denote the intersection of their stabilizers in $\PGL(N, q)$. We denote the neutral element of a group by $id$.

\begin{Lemma}\label{lem:stab_arg}
  An MSR subspace family $\cS$ with $\Stab(\cS) = \{id\}$ is maximal.
\end{Lemma}
\begin{proof}
  Suppose that there exists an $m$-space $S$ such that 
  $\cS \cup \{ S \}$ is an MSR subspace family.
  Then there exists $g \in \Stab(\cS)$
  such that $S^g \cap S$ is trivial. This is a contradiction.
\end{proof}

As $\Stab(\{P_1, P_2, P_3\}) = \{ id \}$ for three pairwise distinct points 
of $\Ff_q^2$, we find that an MSR subspace family of $\Ff_q^2$ has at
most size $3$. Similarly, we find
\begin{Corollary}
  An $(m, 1)$-MSR subspace family of $\Ff_q^m$ has at most size $m+1$.
\end{Corollary}

We will also need the following simple observations.

\begin{Lemma}\label{lem:T_stab_crit}
  For an $(m, r)$-MSR subspace family $\cS$ of $\Ff_q^{2m}$, put $U = \Stab(\cS)$.
  Let $\ell \geq m+1$.
  Suppose that one of the following (equivalent) conditions is satisfied:
  \begin{enumerate}[(a)]
   \item There exists an $\ell$-space $H$ of $\Ff_q^{2m}$ which is fixed point-wise by $U$.
   \item There exists a frame $\cB$ of size $\ell+1$ which is fixed element-wise by $U$.
  \end{enumerate}
  Then $\cS$ is maximal.
\end{Lemma}
\begin{proof}
  Both conditions are equivalent as $\Stab(H)$ acts regularly on frames of $H$.
  Suppose that $\cS$ is not maximal and we find an $m$-space $S$ such that 
  $\cS \cup \{ S \}$ is an MSR subspace family. Then there exists a $g \in \Stab(\cS)$
  such that $S \cap S^g$ is trivial. This contradicts that $(S \cap H)^g = S \cap H$ 
  contains a point.
\end{proof}

\begin{Lemma}\label{lem:no_m_meet}
  For an $(m, r)$-MSR subspace family $\cS$ of $\Ff_q^{rm}$, 
  we have that 
  $\dim(S_1 \cap \ldots \cap S_k) \leq \dim(S_1 \cap \ldots S_{k-1})/m$ 
  for pairwise distinct $S_1, \ldots, S_k \in \cS$.
\end{Lemma}
\begin{proof}
  Put $d := \dim(S_1 \cap \ldots \cap S_{k-1})$.
  Otherwise, $T := S_1 \cap \ldots \cap S_k$ has dimension 
  larger than $d/r$. Hence, $\dim(\< T^{g_i} \>) \leq d$
  for any $g_1, \ldots, g_r \in \Stab(S_1)$.
  Hence, $\cS$ cannot be an MSR subspace family
  as this requires that $\dim(\< T_{g_i} \>) = rd$.
\end{proof}

%

The following is a poor man's version 
of Lemma 7 in \cite{AG2021}. We state it for the 
sake of completeness.

\begin{Lemma}\label{lem:strict_dec}
  For an $(m, r)$-MSR subspace family $\cS$ of $\Ff_q^{rm}$,
  we have that $|\Stab(\cT)| > |\Stab(\cT')|$ 
  for $\cT' \subseteq \cS$ if $\cT$ is a proper subset of $\cT$.
\end{Lemma}
\begin{proof}
  Suppose that $|\cT'| = |\cT|+1$ 
  and let $S$ be the element in $\cT' \setminus \cT$.
  If $\Stab(\cT) = \Stab(\cT')$, then 
  then $S^g = S$ for all $g \in \Stab(\cT)$,
  so $\cS$ cannot be an MSR subspace family.
\end{proof}

\subsection{The Segre Variety}

We will use some basic facts about Segre varieties over finite fields.
We will also give an explicit coordinatized example for $\Ff_q^{2m}$, 
our main interest in this document.
If we have 
three pairwise disjoint $m$-spaces $S_1, S_2, S_3$ in $\Ff_q^{2m}$,
then there exist precisely $[m] := (q^m-1)/(q-1)$ lines 
$\cR^{opp} := \cR^{opp}(S_1, S_2, S_3) := \{ L_1, \ldots, L_{[m]} \}$
which meet $S_1, S_2, S_3$ each in a point.
Furthermore, there exist precisely $q+1$ subspaces 
$\cR := \cR(S_1, S_2, S_3) := \{ S_1, \ldots, S_{q+1}\}$
which meet each element of $\cR^{opp}$ in a point.
We call $\cR$ a \textit{regulus}, and $\cR \cup \cR^{opp}$
is the \textit{Segre variety} $\cS_{m-1,1,q}$.
For $m=2$, we call $\cR^{opp}$
the \textit{opposite regulus} of $\cR$,
and $S_{m-1,1,q}$ the line set of a \textit{hyperbolic quadric}.
Let $\cP(\cR) = \cP(\cR^{opp})$ denote the point set of the lines 
in $\cR$.

The setwise stabilizer of $\cR$ (in $\PGL(2m, q)$) is isomorphic 
to $\PGL(2, q) \times \PGL(m, q)$ 
and acts on $\cR \cup \cR^{opp}$ as expected,
that is one $\PGL(2, q)$ acts on $\cR$ as on points of $\Ff_q^2$,
and one $\PGL(m, q)$ acts on $\cR^{opp}$ as on points of $\Ff_q^m$.
This leads to the following well-known result on $\Stab(\{S_1,S_2,S_3\})$.

\begin{Lemma}\label{lem:fix_segre}
  Let $S_1, S_2, S_3$ be three pairwise disjoint $m$-spaces of $\Ff_q^{2m}$.
  Then $\Stab(\allowbreak\{S_1, S_2, S_3\})$ fixes each element of $\cR = \cR(S_1, S_2, S_3)$
  and 
is isomorphic to $\PGL(m, q)$. \hfill \qed 
\end{Lemma}

\begin{Example}\label{ex:reg}
Let $m=2$.
Without loss of generality (as $\PGL(4, q)$ acts transitively on triples of pairwise 
disjoint lines) $\cR$ consists of the lines
\begin{align*}
  \{ \< e_1 + \alpha e_2, e_3 + \alpha e_4: \alpha \in \Ff_q\> \cup \{ \< e_2, e_4 \> \},
\end{align*}
and $\cR^{opp}$ consists of the lines 
\begin{align*}
  \{ \< e_1 + \alpha e_3, e_2 + \alpha e_4: \alpha \in \Ff_q\> \cup \{ \< e_3, e_4 \> \}.
\end{align*}
Note that this is the canonical tensor product $\Ff_q^2 \otimes \Ff_q^2$,
that is points of the first $\Ff_q^2$ correspond to lines in $\cR$
and the points of the second $\Ff_q^2$ correspond to lines in $\cR^{opp}$.
The setwise stabilizer of $\cR$ in $\PGL(4, q)$ is given by
\begin{align*}
  & \left\{ 
    \begin{pmatrix}
      a\alpha & b\alpha & a\beta & b\beta\\
      c\alpha & d\alpha & c\beta & d\beta\\
      a\gamma & b\gamma & a\delta & b\delta\\
      c\gamma & d\gamma & c\delta & d\delta\\
    \end{pmatrix}
  : \begin{pmatrix} a & b \\ c & d \end{pmatrix},
  \begin{pmatrix}
   \alpha & \beta \\ \gamma & \delta
  \end{pmatrix}
\in \GL(2, q) \right\}/\Ff_q^*,,
\end{align*}
that is $\PGL(2, q) \times \PGL(2, q)$.
Hence, 
\begin{align*}
  \Stab(\{S_1, S_2, S_3 \}) &= \left\{ 
    \begin{pmatrix}
      \alpha & 0 & \beta & 0\\
      0 & \alpha & 0 & \beta\\
      \gamma & 0 & \delta & 0\\
      0 & \gamma & 0 & \delta\\
    \end{pmatrix}: \begin{pmatrix} \alpha & \beta \\ \gamma & \delta \end{pmatrix} \in \GL(2, q)
  \right\}/\Ff_q^* \\
  &= 
  \left\{ 
    \begin{pmatrix}
      \alpha I & \beta I\\
      \gamma I & \delta I
    \end{pmatrix}: \begin{pmatrix} \alpha & \beta \\ \gamma & \delta \end{pmatrix} \in \GL(2, q)
  \right\}/\Ff_q^*,
\end{align*}
and, clearly, $\Stab(\{S_1, S_2, S_3 \})$ is isomorphic to $\PGL(2, q)$.
\end{Example}

Note that $\cP(\cR)$ corresponds to the points $\<x\>$ with $Q(x)=0$
on a nondegenerate quadratic form $Q$ of rank $2$,
for instance the points $\cP(\cR)$ in Example \ref{ex:reg} are precisely the 
points which vanish on $Q(x) = x_1x_4 + x_2x_3$.
We find the following types of lines with respect to $\cR$:
\begin{enumerate}[(L1)]
 \item A line of $\cR$.
 \item A line of $\cR^{opp}$.
 \item A line, say $\<e_1, e_2\>$, which meet $\cP(\cR)$ in two points (secants), say $Q(x) = x_1x_2$.
 In this example the points of $\cP(\cR)$ are $\<x_1\>$ and $\<x_2\>$.
 \item A line, say $\<e_1,e_2\>$, which meet $\cP(\cR)$ in one point (tangents), say $Q(x) = x_1^2$.
 In this example the point of $\cP(\cR)$ is $\<x_2\>$.
 \item A line, say $\< e_1, e_2\>$, which meet $\cP(\cR)$ in no point (passants), say $Q(x) = x_1^2 + \alpha x_1x_2 + \beta x_2^2$
 such that $1+\alpha x_2 + \beta x_2^2$ is irreducible over $\Ff_q$.
\end{enumerate}
And the following types of planes:
\begin{enumerate}[(P1)]
 \item A plane, say $\< e_1, e_2, e_3\>$, intersects $\cP(\cR)$ in conic (conic plane), that is there exists 
         a nondegenerate quadratic form $Q$, say $Q(x) = x_1^2 + x_2x_3$, 
         such that $\cP(\cR)$ is the set of points $\<x\>$
         of the plane with $Q(x)=0$.
 \item A plane, say $\< e_1, e_2, e_3\>$, intersects $\cP(\cR)$ in two lines (degenerate plane),
    that is there exists a degenerate quadratic form $Q$, say $Q(x) = x_2x_3$.
    In the example the two lines are $\<e_1, e_2\>$ and $\<e_1, e_3\>$.
\end{enumerate}

Note that for a passant $L$, all $q+1$ planes through $L$
are conic planes, while for a line $L$ of $\cR \cup \cR^{opp}$
all planes through $L$ are degenerate planes.
The group $\Stab(\cR)$ has precisely five orbits on lines 
and two orbits on planes as given above. 
All $q+1$ planes through a passant are conic planes, 
all $q+1$ planes through a secant are degenerate planes.

\subsection{\texorpdfstring{Case $r=2$}{Case r=2}}

Here we assume that $\cS = \{ S_1, \ldots, S_k \}$ 
is an MSR subspace family of $m$-spaces.
Lemma \ref{lem:stab_arg} and Lemma \ref{lem:fix_segre} 
imply that $|\cR(S_1, S_2, S_3) \cap \cS| \leq 3$.
Lemma 1 of \cite{LZ2016} show the following.

\begin{Lemma}[Lavrauw \& Zanella]\label{prop:lz}
  Let $q \geq m$.
  Let $\cR$ be a regulus of $\Ff_q^{2m}$ and let 
  $S$ be a subspace disjoint from the $\cR$.
  Then each $(n+1)$-space $T$ of $\Ff_q^{2m}$ 
  through $S$ intersects the point set of $\cR$
  in a normal rational curve.
\end{Lemma}

\begin{Lemma}\label{lem:all_disj_gen}
  Let $q \geq m+1 \geq 3$.
  Suppose that $S_1, S_2, S_3$ are pairwise disjoint
  and let $\cP$ denote the point set of $\cR(S_1, S_2, S_3)$.
  If $S_4$ lies in an $(n+1)$-space $T$ such that
  $T = \< T \cap \cP \>$,
  then $|\cS| \leq 5$.
\end{Lemma}
\begin{proof}
  Let $\cR = \cR(S_1, S_2, S_3) = \{ R_1, \ldots, R_{q+1} \}$.
  By Lemma \ref{prop:lz}, $T \cap \cP$ is a rational 
  normal curve.
  Take a second (not necessarily distinct) 
  $(m+1)$-spaces $T'$ through $S_4$
  such that $T' = \< T' \cap \cP \>$ is a rational normal curve.
  Consider $U := \Stab(\{ S_1, S_2, S_3, S_4 \})$.
  Put $P_i = R_i \cap T$ and $P_i' = R_i \cap T'$ 
  for $1 \leq i \leq m+2$. 
  By Lemma \ref{prop:lz},
  the $P_i$, respectively, $P_i'$ are a frame 
  of $T$, respectively, of $T'$.
  We find at most one map with $P_i^g = P_i'$ for 
  all $1 \leq i \leq m+2$ as $\PGL(m+1, q)$
  acts regularly on frames of $\Ff_q^{m+1}$.
  By Lemma \ref{lem:strict_dec},
  $U' := \Stab(\{ S_1, S_2, S_3, S_4, S_5 \})$
  will fix one of the $(m+1)$-spaces $T$ through 
  $S_4$. Hence, $T$ is fixed point-wise by $U'$.
  By Lemma \ref{lem:stab_arg}, $\cS$ is maximal.
\end{proof}

Note that $T$ always exists for $m \in \{ 2, 3 \}$, but Theorem 6
and Corollary 7 in \cite{LZ2016} show that for $m$
not a prime there always exist an $n$-space $S_4$ disjoint from $\cR$
such that $\< T \cap \cP\>$ is a proper subspace of $T$
for all $(n+1)$-spaces $T$ through $S_4$. 

%

\subsection{\texorpdfstring{Case $r=m=2$}{Case r=m=2}}

From here on we assume that $m=2$.
For $q=2$, we have that $k \leq 4$ as can be
easily verified, for instance by 
going through all $\binom{35}{5} = 324\, 632$
$5$-sets of lines in $\Ff_2^4$ by computer.
Therefore, we assume that $q \geq 3$.
By Lemma \ref{lem:all_disj_gen},
we know that $|\cS| \leq 5$ if $\cS$ 
contains four pairwise disjoint lines.

\begin{Lemma}\label{lem:some_disj}
  Suppose that $k \geq 5$.
  If $S_1,S_2,S_3$ are pairwise disjoint and $S_4$
  meets $\cR = \cR(S_1, S_2, S_3)$ in at least a point, then $|\cS| \leq 6$ 
  with equality if and only if $S_4, S_5, S_6 \in \cR^{opp}$.
\end{Lemma}
\begin{proof}
  First we assume that $S_5$ and $S_6$ (if it exists) also meet $\cP(\cR)$ in a point.
  Then $S_i$ meets $\cP(\cR)$ in a point on a 
  line $L_i \in \cR^{opp}$ for $i \in \{ 4,5,6\}$.
  Note that $L_i$ is not necessarily unique as $S_i$ can meet $\cR$ in two points
  if it is a secant.
  Also note that in this case, that is $S_i$ meets $\cP(\cR)$ in points $Q_i$ and $Q_i'$
  on lines $L_i$ and $L_i'$, then there exists no $g \in \Stab(\cR)$ such that 
  $Q_i^g = Q_i'$ (as $Q_i$ and $Q_i'$ do not lie on the same line of $\cR$
  and $\Stab(\cR)$ fixes each line of $\cR$ point-wise).
  Hence, $\Stab(\{ S_1, S_2, S_3, S_i \})$ contains $\Stab(\cR \cup \{ L_i \})$.
  By Lemma \ref{lem:fix_segre}, $\Stab(\cR \cup \{ L_i \})$ fixes the points of $L_i$ for $i \in \{ 4, 5\}$.
  
  {\medskip 
  More explicitly, without loss of generality $L_4 = \< e_1, e_2 \>$ and $L_5 = \< e_3, e_4 \>$ (as $\Stab(\cR)$ acts transitively on triples of pairwise distinct elements of $\cR^{opp}$).
  Then
  \begin{align*}
    &\Stab(\cR \cup \{ L_4 \}) = \left\{ 
    \begin{pmatrix}
      aI & bI\\
      0 & dI
    \end{pmatrix}: \begin{pmatrix} a & b \\ 0 & d \end{pmatrix} \in \GL(2, q)
  \right\}/\Ff_q^*, \\
  &
    \Stab(\cR \cup \{ L_5 \}) = \left\{ 
    \begin{pmatrix}
      aI & 0\\
      cI & dI
    \end{pmatrix}: \begin{pmatrix} a & 0 \\ c & d \end{pmatrix} \in \GL(2, q)
  \right\}/\Ff_q^*. 
  \end{align*} \par} 
  
  \medskip
  
  In particular, this implies that $L_4 \neq L_5$ as 
  otherwise $(S_5 \cap L_4)^{g_5} = S_5 \cap L_4$.
  Suppose that $S_4 \neq L_4$. Put $H = \< L_4, S_4\>$. Then 
  $U := \Stab(\{ S_1, S_2, S_3, S_4, S_5 \})$ fixes each point of $L_4$, 
  the line $S_4$, and $H \cap L_5$. 
  Hence, $U$ fixes a frame $P_1 := L_4 \cap S_4$, $P_2 \subseteq L_4 \setminus \{ P_1 \}$,
  $P_3 := L_5 \cap H$, $P_4 := \< P_2', P_3\> \cap S_4$, where $P_2' \subseteq L_4 \setminus \{ P_1, P_2 \}$.
  Hence, $U$ fixes $H$ point-wise. By Lemma \ref{lem:T_stab_crit},
  $|\cS| \leq 5$.
  As we only assumed that $S_4 \neq L_4$, that is $S_4 \notin \cR^{opp}$, 
  we have $|\cS| \leq 5$ as long as $S_i \neq L_i$ for any $i \in \{ 4, 5, 6 \}$.
  Otherwise, $S_4, S_5, S_6 \in \cR^{opp}$ and, by Lemma \ref{lem:fix_segre},
  $\Stab(\{ S_1, S_2, S_3, S_4, S_5, S_6\} = \{ id \}$.
  By Lemma \ref{lem:stab_arg}, we are done.
  
  \medskip 
  More explicitly, if $L_4$ and $L_5$ are chosen as before, then $U$ is contained in
  \begin{align*}
    \Stab(\cR \cup \{ L_4, L_5 \}) = \left\{
    \begin{pmatrix}
      aI & 0\\
      0 & dI
    \end{pmatrix}: \begin{pmatrix} a & 0 \\ 0 & d \end{pmatrix} \in \GL(2, q)
  \right\}/\Ff_q^*.
  \end{align*}
  Hence, the claim in the preceding paragraph is also easily verified with an explicit 
  calculation.
\end{proof}

In Example \ref{ex:constr}.2, the lines $S_1, S_2, S_3$ lie in 
the regulus $\cR = \{ \< e_1 + \beta e_2, e_3 + \beta e_4 \>: 
\beta \in \Ff_q \} \cup \{ \< e_2, e_4 \> \}$,
and the lines $S_4, S_5, S_6$ lie in the opposite regulus
$\cR = \{ \< e_1 + \beta e_3, e_2 + \beta e_4 \>: 
\beta \in \Ff_q \} \cup \{ \< e_3, e_4 \> \}$.
Note that under the natural action of $\PGL(2, q) \times \PGL(2, q)$
on $\cR$ all choices for a set of six lines with three in $\cR$
and three in $\cR^{opp}$ are isomorphic.

\begin{Lemma}\label{lem:no_disj}
  If $\cS$ does not contain a triple of pairwise disjoint subspaces, then 
  $|\cS| \leq 4$.
\end{Lemma}
\begin{proof}
  Recall that by Lemma \ref{lem:no_m_meet} no three 
  pairwise distinct elements of $\cS$ meet in a point.
  Suppose that we find a quadrangle $S_1, S_2, S_3, S_4$ in $\cS$, that is 
  without loss of generality $S_1$ and $S_3$, respectively, $S_2$ and $S_4$
  are pairwise disjoint and $S_i \cap S_{i+1}$ are points (with $S_{4+1}$ read as $S_1$).
  If $|\cS| \geq 5$, then, as there are no three pairwise disjoint subspaces in $\cS$,
  either $S_5 = \< S_1 \cap S_2, S_3 \cap S_4 \>$ or $S_5 = \< S_1 \cap S_4, S_3 \cap S_2 \>$, a contradiction.
  
  If we do not find a quadrangle, and $|\cS| \geq 4$,
  then we find without loss of generality that $S_1, S_2, S_3$
  are coplanar, namely in $\< S_1, S_2 \>$. But then 
  $S_3^g \subseteq \< S_1, S_2 \>$ for all $g \in \Stab(\{ S_1, S_2 \})$,
  so $S_3 \cap S_3^g$ is nontrivial. Hence, $|\cS| \leq 2$.
\end{proof}

Lemma \ref{lem:all_disj_gen}, Lemma \ref{lem:some_disj}, and Lemma \ref{lem:no_disj}
together show Theorem \ref{thm:MSR_n_2}.

\subsection{\texorpdfstring{Case $r=3$ and $m=2$}{Case r=3 and m=2}}
\label{sec:rthree}

For a set of matrices $\cM$ over $\Ff_q$, let 
$\dim(\cM)$ denote the dimension of the smallest 
subspace which contains $\cM$.
By Lemma 7 of \cite{AG2021}, 
\begin{align}
 \dim(\Stab(S_1, \ldots, S_i)) \leq \frac{r^2-r+1}{r^2} \cdot \dim(\Stab(S_1, \ldots, S_{i-1})). \label{eq:imp}
\end{align}
Applying Equation \ref{eq:imp} iteratively, 
taking into account that $\dim(\Stab(S_1, \ldots, S_k))$
is always an integer, gives an upper bound of $|\cS| \leq 7$ for $r=m=2$.

Lemma \ref{lem:all_disj_gen} can be used to improve the bounds 
in \cite{AG2021} for specific dimensions.
For $(r,m) = (3, 2)$, 
iteratively applying Equation \eqref{eq:imp} 
directly only shows that $|\cS| \leq 10$.
If $S_1$ and $S_2$ meet in precisely a point,
say $S_1 = \< e_1, e_2, e_3 \>$ and $S_2 = \< e_1, e_4, e_5 \>$,
then we see that $\dim(\allowbreak \Stab(S_1, S_2)) = 1 + 2 \cdot 3 + 2 \cdot 3 + 6=19$.
Iteratively applying Equation \ref{eq:imp} shows that 
$|\cS| \leq 9$ in this case. By Lemma \ref{lem:no_m_meet},
no two elements of $\cS$ can meet in more than a point.
It can be checked that Lemma \ref{lem:all_disj_gen}
always applies, so $|\cS| \leq 5$ if no two elements of 
$\cS$ meet nontrivially. Hence,

\begin{Proposition}\label{thm:MSR_n_3}
  A $(6,3)$-MSR subspace family has size at most $9$.
\end{Proposition}

We believe that a more refined analysis in the spirit of 
this document can classify the largest $(6,3)$- and 
$(8,4)$-MSR subspace families.

%

\bigskip
\paragraph*{Acknowledgments}
The author thanks Sascha Kurz, John Sheekey, Itzhak Tamo, and the referee 
for their helpful comments, particularly
Sascha Kurz for introducing him to the problem.
The author was supported by a 
postdoctoral fellowship of the Research Foundation -- Flanders (FWO)
during the work on this document.

\end{document}